\documentclass[a4paper,12pt,fleqn]{article}

\usepackage{amscd,amsmath,amssymb,amsthm}
\usepackage[pdfpagemode=UseNone,pdfstartview=FitH,hypertex]{hyperref}
\usepackage[all]{xy} \CompileMatrices

\newcommand{\bdry}[1]{\partial #1}
\newcommand{\bgset}[1]{\big\{#1\big\}}
\newcommand{\closure}[1]{\overline{#1}}
\newcommand{\dint}{\ds{\int}}
\newcommand{\ds}[1]{\displaystyle #1}
\newcommand{\eps}{\varepsilon}
\newcommand{\half}{\frac{1}{2}}
\newcommand{\homo}{\simeq}
\newcommand{\hquad}{\hspace{0.08in}}
\newcommand{\isom}{\approx}
\newcommand{\norm}[2][]{\left\|#2\right\|_{#1}}
\renewcommand{\o}{\text{o}}
\newcommand{\PS}[1]{$(\text{PS})_{#1}$}
\newcommand{\pnorm}[2][]{\if #1'' \left|#2\right|_p \else \left|#2\right|_{#1} \fi}
\newcommand{\R}{\mathbb R}
\newcommand{\seq}[1]{\left(#1\right)}
\newcommand{\set}[1]{\left\{#1\right\}}

\DeclareMathOperator{\im}{im}

\newenvironment{enumroman}{\begin{enumerate}

}{\end{enumerate}}

\newtheorem{lemma}{Lemma}[section]
\newtheorem{proposition}[lemma]{Proposition}
\newtheorem{theorem}[lemma]{Theorem}

\numberwithin{equation}{section}

\title{\bf Pairs of nontrivial solutions to concave-linear-convex type elliptic problems}
\author{\bf Pasquale Candito\\
DICEAM, Universit\`{a} degli Studi di Reggio Calabria\\
89100 Reggio Calabria, Italy\\
\em pasquale.candito@unirc.it\\
[\bigskipamount]
\bf Salvatore A. Marano\\
Dipartimento di Matematica e Informatica\\
Universit\`{a} degli Studi di Catania\\
95125 Catania, Italy\\
\em marano@dmi.unict.it\\
[\bigskipamount]
\bf Kanishka Perera\thanks{This work was completed while the third-named author was visiting Universit\`{a} degli Studi di Catania, and he is grateful for the kind hospitality of the host university.}\\
Department of Mathematical Sciences\\
Florida Institute of Technology\\
Melbourne, FL 32901, USA\\
\em kperera@fit.edu}
\date{}

\begin{document}

\maketitle

\begin{abstract}
We obtain a pair of nontrivial solutions for a class of concave-linear-convex type elliptic problems that are either critical or subcritical. The solutions we find are neither local minimizers nor of mountain pass type in general. They are higher critical points in the sense that they each have a higher critical group that is nontrivial. This fact is crucial for showing that our solutions are nontrivial. We also prove some intermediate results of independent interest on the localization and homotopy invariance of critical groups of functionals involving critical Sobolev exponents.
\end{abstract}

\section{Introduction}

Consider the concave-linear-convex type problem
\begin{equation} \label{1}
\left\{\begin{aligned}
- \Delta u & = \mu f(x,u) + \lambda u + |u|^{p-2}\, u && \text{in } \Omega\\[10pt]
u & = 0 && \text{on } \bdry{\Omega},
\end{aligned}\right.
\end{equation}
where $\Omega$ is a bounded domain in $\R^N$, $N \ge 2$, $p > 2$ if $N = 2$ and $2 < p \le 2^\ast = 2N/(N - 2)$ if $N \ge 3$, $\lambda \ge 0$ and $\mu > 0$ are parameters, and $f$ is a Carath\'{e}odory function on $\Omega \times \R$ satisfying
\begin{equation} \label{3}
f(x,t) = |t|^{\sigma - 2}\, t + \o(|t|^{\sigma - 1}) \quad \text{as } t \to 0
\end{equation}
uniformly a.e. in $\Omega$ for some $\sigma \in (1,2)$ and
\begin{equation} \label{2}
|f(x,t)| \le C(|t|^{r-1} + 1) \quad \forall (x,t) \in \Omega \times \R
\end{equation}
for some $C > 0$ and $2 < r < p$. Ambrosetti et al.\! studied the special case $\lambda = 0$ of this problem in the pioneering paper \cite{MR1276168}, and showed, among other things, that there are two positive solutions for all sufficiently small $\mu > 0$. The first solution is a local minimizer of the associated energy functional
\[
E_0(u) = \int_\Omega \left(\half\, |\nabla u|^2 - \mu F(x,u) - \frac{1}{p}\, |u|^p\right) dx, \quad u \in H^1_0(\Omega),
\]
where $F(x,t) = \int_0^t f(x,s)\, ds$, and the second solution is a mountain pass point.

Their result is easily extended to the case $0 < \lambda < \lambda_1$, where $\lambda_1 > 0$ is the first Dirichlet eigenvalue of $- \Delta$ on $\Omega$. When $\lambda \ge \lambda_1$, we cannot expect to find positive solutions in general, but we may ask if the problem still has a pair of nontrivial solutions for small $\mu > 0$. In the present paper we show that this is indeed the case if
\begin{equation} \label{37}
F(x,t) \ge 0 \quad \forall (x,t) \in \Omega \times \R
\end{equation}
in the following cases:
\begin{enumroman}
\item $N = 2$ and $p > 2$,
\item $N = 3$ and $2 < p < 6$,
\item $N \ge 4$ and $2 < p \le 2^\ast$.
\end{enumroman}
However, the solutions we find here are neither local minimizers nor of mountain pass type in general. They are higher critical points in the sense that they each have a higher critical group that is nontrivial. This fact is crucial to showing that our solutions are themselves nontrivial.

First we have the following theorem in the critical case $p = 2^\ast$.

\begin{theorem} \label{Theorem 1}
If $N \ge 4$, \eqref{3}--\eqref{37} hold, and $\lambda \ge \lambda_1$, then there exists $\mu^\ast > 0$ such that the problem
\begin{equation} \label{49}
\left\{\begin{aligned}
- \Delta u & = \mu f(x,u) + \lambda u + |u|^{2^\ast - 2}\, u && \text{in } \Omega\\[10pt]
u & = 0 && \text{on } \bdry{\Omega}
\end{aligned}\right.
\end{equation}
has two nontrivial solutions for all $\mu \in (0,\mu^\ast)$.
\end{theorem}

In the subcritical case we replace the linear term $\lambda u$ with a more general nonlinearity $g(x,u)$ and consider the problem
\begin{equation} \label{4}
\left\{\begin{aligned}
- \Delta u & = \mu f(x,u) + g(x,u) + |u|^{p-2}\, u && \text{in } \Omega\\[10pt]
u & = 0 && \text{on } \bdry{\Omega},
\end{aligned}\right.
\end{equation}
where $p > 2$ if $N = 2$ and $2 < p < 2^\ast$ if $N \ge 3$, $\mu > 0$ is a parameter, $f$ satisfies \eqref{3}--\eqref{37}, and $g$ is a Carath\'{e}odory function on $\Omega \times \R$ satisfying
\begin{equation} \label{40}
g(x,t) = \o(|t|^{\sigma - 1}) \quad \text{as } t \to 0
\end{equation}
uniformly a.e. in $\Omega$ and
\begin{equation} \label{5}
|g(x,t)| \le C(|t|^{r-1} + 1) \quad \forall (x,t) \in \Omega \times \R.
\end{equation}
Let $G(x,t) = \int_0^t g(x,s)\, ds$, and let $\lambda_1 < \lambda_2 \le \lambda_3 \le \cdots$ be the Dirichlet eigenvalues of $- \Delta$ on $\Omega$, repeated according to multiplicity. We assume that
\begin{equation} \label{10}
G(x,t) \ge \half\, \lambda_l\, t^2 \quad \forall (x,t) \in \Omega \times \R
\end{equation}
for some $l \ge 1$ and
\begin{equation} \label{9}
G(x,t) \le \half\, \lambda\, t^2 \quad \forall x \in \Omega,\, |t| \le \delta
\end{equation}
for some $\delta > 0$ and $\lambda < \lambda_{l+1}$. We have the following theorem.

\begin{theorem} \label{Theorem 2}
If \eqref{3}--\eqref{37} and \eqref{40}--\eqref{9} hold, then there exists $\mu^\ast > 0$ such that problem \eqref{4} has two nontrivial solutions for all $\mu \in (0,\mu^\ast)$.
\end{theorem}

Although Theorems \ref{Theorem 1} and \ref{Theorem 2} apply to the model case $f(x,t) = |t|^{\sigma - 2}\, t$, $f$ is not assumed to be odd in $t$ in these theorems, so symmetry does not play a role here. We also note that \eqref{40}--\eqref{9} hold for $g(x,t) = \lambda t$ when $\lambda_l \le \lambda < \lambda_{l+1}$.

Weak solutions of problem \eqref{4} coincide with critical points of the $C^1$-functional
\[
E(u) = \int_\Omega \left(\half\, |\nabla u|^2 - \mu F(x,u) - G(x,u) - \frac{1}{p}\, |u|^p\right) dx, \quad u \in H^1_0(\Omega).
\]
Recall that critical groups of $E$ at an isolated critical point $u_0$ are defined by
\[
C_q(E,u_0) = H_q(E^c \cap U,E^c \cap U \setminus \set{u_0}), \quad q \ge 0,
\]
where $c = E(u_0)$, $E^c = \set{u \in H^1_0(\Omega) : E(u) \le c}$, $U$ is a neighborhood of $u_0$, and $H_\ast$ denotes singular homology. First we will show that $E$ has of critical points $u_1, u_2$ with
\[
C_l(E,u_1) \ne 0, \qquad C_{l+1}(E,u_2) \ne 0
\]
if $\mu > 0$ is sufficiently small. This will be based on the following abstract result adapted from Perera \cite{MR1473858}.

\begin{theorem} \label{Theorem 3}
Let $E$ be a $C^1$-functional on a Banach space $X$. Assume that there are a direct sum decomposition $X = Y \oplus Z,\, u = v + w$, with $l=\dim Y < \infty$, $z_0 \in X \setminus Y$, $0 < \rho < R$, and $a < b$ such that, setting
\[
A = \set{u = v + t z_0 : v \in Y,\, t \ge 0,\, \norm{u} \le R}, \qquad B = \set{w \in Z : \norm{w} \le \rho}
\]
and denoting by $\bdry{A}$ and $\bdry{B}$ the relative boundaries of $A$ and $B$, respectively, we have
\[
a < \inf_B\, E, \qquad \sup_{\bdry{A}}\, E < \inf_{\bdry{B}}\, E, \qquad \sup_A\, E < b.
\]
Assume further that $E$ satisfies the {\em \PS{c}} condition for all $c \in (a,b)$ and has only a finite number of critical points in $E^{-1}((a,b))$. Then $E$ has a pair of critical points $u_1, u_2$ with
\[
\inf_B\, E \le E(u_1) \le \sup_{\bdry{A}}\, E, \qquad \inf_{\bdry{B}}\, E \le E(u_2) \le \sup_A\, E
\]
and
\[
C_l(E,u_1) \ne 0, \qquad C_{l+1}(E,u_2) \ne 0.
\]
\end{theorem}

To show that $u_1$ and $u_2$ are nontrivial and complete the proof of Theorem \ref{Theorem 2}, we will show that $C_q(E,0) = 0$ for all $q \ge 0$. This will be based on two intermediate results of independent interest on the critical groups of the functional
\[
E(u) = \int_\Omega \left(\half\, |\nabla u|^2 - H(x,u) - \frac{1}{p}\, |u|^p\right) dx, \quad u \in H^1_0(\Omega)
\]
at $0$, where $p > 2$ if $N = 2$ and $2 < p \le 2^\ast$ if $N \ge 3$, $H(x,t) = \int_0^t h(x,s)\, ds$, and $h$ is a Carath\'{e}odory function on $\Omega \times \R$ satisfying
\begin{equation} \label{11}
|h(x,t)| \le C(|t|^{r-1} + 1) \quad \forall (x,t) \in \Omega \times \R
\end{equation}
for some $C > 0$ and $2 < r < p$. The first is the following localization result.

\begin{theorem} \label{Theorem 6}
Assume that \eqref{11} holds. Let $\delta > 0$, let $\vartheta : \R \to [- \delta,\delta]$ be a smooth nondecreasing function such that $\vartheta(t) = - \delta$ for $t \le - \delta$, $\vartheta(t) = t$ for $- \delta/2 \le t \le \delta/2$, and $\vartheta(t) = \delta$ for $t \ge \delta$, and set
\[
\widetilde{E}(u) = \int_\Omega \left(\half\, |\nabla u|^2 - H(x,\vartheta(u)) - \frac{1}{p}\, |\vartheta(u)|^p\right) dx, \quad u \in H^1_0(\Omega).
\]
If $0$ is an isolated critical point of $E$, then it is also an isolated critical point of $\widetilde{E}$ and
\[
C_q(E,0) \isom C_q(\widetilde{E},0) \quad \forall q \ge 0.
\]
\end{theorem}

Henceforth, $\isom$ will denote the group isomorphism. This result is somewhat surprising given that $H^1_0(\Omega)$ is not embedded in $L^\infty(\Omega)$ when $N \ge 2$. Our second intermediate result is the following critical group computation.

\begin{theorem} \label{Theorem 4}
Assume that \eqref{11} holds and
\begin{equation} \label{13}
h(x,t) = \mu\, |t|^{\sigma - 2}\, t + \o(|t|^{\sigma - 1}) \quad \text{as } t \to 0
\end{equation}
uniformly a.e. in $\Omega$ for some $\mu > 0$ and $\sigma \in (1,2)$. If $0$ is an isolated critical point of $E$, then
\[
C_q(E,0) = 0 \quad \forall q \ge 0.
\]
\end{theorem}

The rest of this paper is organized as follows. The preliminary results in Theorems \ref{Theorem 3}--\ref{Theorem 4} will be proved in the next section. Our main results in Theorems \ref{Theorem 1} and \ref{Theorem 2} are proved in Section \ref{Proofs}.

\section{Proofs of intermediate results}

Let $E$ be a $C^1$-functional on a Banach space $X$. We recall that $E$ satisfies the \PS{c} condition if every sequence $\seq{u_j} \subset X$ such that $E(u_j) \to c$ and $E'(u_j) \to 0$, called a \PS{c} sequence of $E$, has a convergent subsequence. The proof of Theorem \ref{Theorem 3} is based on the following well-known result from Morse theory (see Chang \cite[Chapter II, Theorem 1.5]{MR94e:58023}).

\begin{proposition} \label{Proposition 1}
Assume that $H_l(E^b,E^a) \ne 0$, where $a < b$ are regular values of $E$ and $l \ge 0$. If $E$ satisfies the {\em \PS{c}} condition for all $c \in [a,b]$ and has only a finite number of critical points in $E^{-1}((a,b))$, then $E$ has a critical point $u_0$ with
\[
a < E(u_0) < b, \qquad C_l(E,u_0) \ne 0.
\]
\end{proposition}

We now prove Theorem \ref{Theorem 3}. For the sake of simplicity we only consider the case where $X$ is infinite dimensional.

\begin{proof}[Proof of Theorem \ref{Theorem 3}]
Since $B \cap \bdry{A}$ and $\bdry{B} \cap A$ are nonempty,
\[
\inf_B\, E \le \sup_{\bdry{A}}\, E, \qquad \inf_{\bdry{B}}\, E \le \sup_A\, E.
\]
We will show that if $\alpha < \beta < \gamma$ satisfy
\[
a < \alpha < \inf_B\, E, \qquad \sup_{\bdry{A}}\, E < \beta < \inf_{\bdry{B}}\, E, \qquad \sup_A\, E < \gamma < b,
\]
then
\begin{equation} \label{14}
H_l(E^\beta,E^\alpha) \ne 0, \qquad H_{l+1}(E^\gamma,E^\beta) \ne 0.
\end{equation}
Proposition \ref{Proposition 1} then gives a pair of critical points $u_1, u_2$ of $E$ with
\[
\alpha < E(u_1) < \beta, \qquad \beta < E(u_2) < \gamma
\]
and
\[
C_l(E,u_1) \ne 0, \qquad C_{l+1}(E,u_2) \ne 0
\]
if, in addition, $\alpha$, $\beta$, and $\gamma$ are regular values of $E$. Since $E$ has only a finite number of critical values in $(a,b)$, $\alpha$ and $\beta$ can be chosen so that $E$ has no critical values in $[\alpha,\inf E(B))$ or $(\sup E(\bdry{A}),\beta]$, and hence
\[
\inf_B\, E \le E(u_1) \le \sup_{\bdry{A}}\, E.
\]
Similarly, $u_2$ can be chosen to satisfy
\[
\inf_{\bdry{B}}\, E \le E(u_2) \le \sup_A\, E.
\]

To prove \eqref{14}, recall that $\bdry{A}$ homologically links $\bdry{B}$ in dimension $l$ in the sense that the inclusion $\xi : \bdry{A} \hookrightarrow X \setminus \bdry{B}$ induces a nontrivial homomorphism
\[
\begin{CD}
\xi_\ast : \widetilde{H}_l(\bdry{A}) @>>> \widetilde{H}_l(X \setminus \bdry{B})
\end{CD}
\]
between reduced homology groups (see, e.g., Perera and Schechter \cite[Propositions 2.4.2 and 2.4.4]{MR3012848}). Since $E < \beta$ on $\bdry{A}$ and $E > \beta$ on $\bdry{B}$, we also have inclusions $\varphi : \bdry{A} \hookrightarrow E^\beta$ and $\psi : E^\beta \hookrightarrow X \setminus \bdry{B}$, which induce the commutative diagram
\[
\xymatrix{{\widetilde{H}}_l(\bdry{A}) \ar[r]^{\varphi_\ast} \ar[rd]_{\xi_\ast} & {\widetilde{H}}_l(E^\beta) \ar[d]^{\psi_\ast}\\
& {\widetilde{H}}_l(X \setminus \bdry{B}).}
\]
This gives
\[
\psi_\ast\, \varphi_\ast = \xi_\ast \ne 0,
\]
so both $\varphi_\ast$ and $\psi_\ast$ are nontrivial homomorphisms.

First we show that $H_l(E^\beta,E^\alpha) \ne 0$. Since $E > \alpha$ on $B$ and $\alpha < \beta$, we have the inclusions $E^\alpha \hookrightarrow X \setminus B \hookrightarrow X \setminus \bdry{B}$ and $E^\alpha \hookrightarrow E^\beta \hookrightarrow X \setminus \bdry{B}$, which give the commutative diagram
\[
\begin{CD}
\widetilde{H}_l(E^\alpha) @>i_\ast>> \widetilde{H}_l(E^\beta)\\
@VVV @VV{\psi_\ast}V\\
\widetilde{H}_l(X \setminus B) @>>> \widetilde{H}_l(X \setminus \bdry{B}).
\end{CD}
\]
Since $B$ is contractible, $X \setminus B$ has the homotopy type of the unit sphere in $X$, which is itself contractible since $X$ is infinite dimensional, so $\widetilde{H}_l(X \setminus B) = 0$. This together with the commutativity of the square gives
\[
\psi_\ast\, i_\ast = 0.
\]
Since $\psi_\ast$ is nontrivial, this implies that $i_\ast$ is non-surjective. Now consider the exact sequence of the pair $(E^\beta,E^\alpha)$:
\[
\begin{CD}
\cdots @>\partial_\ast>> \widetilde{H}_l(E^\alpha) @>i_\ast>> \widetilde{H}_l(E^\beta) @>j_\ast>> H_l(E^\beta,E^\alpha) @>\partial_\ast>> \cdots.
\end{CD}
\]
By exactness,
\[
\ker j_\ast = \im i_\ast \ne \widetilde{H}_l(E^\beta),
\]
so the relative homology group $H_l(E^\beta,E^\alpha) \ne 0$.

Finally we show that $H_{l+1}(E^\gamma,E^\beta) \ne 0$. Since $E < \gamma$ on $A$ and $\beta < \gamma$, we have the inclusions $\bdry{A} \hookrightarrow A \hookrightarrow E^\gamma$ and $\bdry{A} \hookrightarrow E^\beta \hookrightarrow E^\gamma$, yielding the commutative diagram
\[
\begin{CD}
\widetilde{H}_l(\bdry{A}) @>\varphi_\ast>> \widetilde{H}_l(E^\beta)\\
@VVV @VV{i_\ast}V\\
\widetilde{H}_l(A) @>>> \widetilde{H}_l(E^\gamma).
\end{CD}
\]
Since $A$ is contractible and hence $\widetilde{H}_l(A) = 0$, this gives
\[
i_\ast\, \varphi_\ast = 0.
\]
Since $\varphi_\ast$ is nontrivial, this implies that $i_\ast$ is non-injective. Now consider the exact sequence of the pair $(E^\gamma,E^\beta)$:
\[
\begin{CD}
\cdots @>j_\ast>> H_{l+1}(E^\gamma,E^\beta) @>\partial_\ast>> \widetilde{H}_l(E^\beta) @>i_\ast>> \widetilde{H}_l(E^\gamma) @>j_\ast>> \cdots.
\end{CD}
\]
By exactness,
\[
\im \partial_\ast = \ker i_\ast \ne 0,
\]
so $H_{l+1}(E^\gamma,E^\beta) \ne 0$.
\end{proof}

Next we turn to the proof of Theorem \ref{Theorem 6}, which will be based on the following result on the invariance of critical groups under homotopies that preserve the isolatedness of the critical point (see Chang and Ghoussoub \cite{MR1422006} or Corvellec and Hantoute \cite{MR1926378}).

\begin{proposition} \label{Proposition 2}
Let $E_\tau,\, \tau \in [0,1]$ be a family of $C^1$-functionals on a Banach space $X$, and let $u_0$ be a critical point of each $E_\tau$. Assume that there exists a closed neighborhood $U$ of $u_0$ such that
\begin{enumroman}
\item the map $[0,1] \to C^1(U,\R),\, \tau \mapsto E_\tau$ is continuous,
\item $U$ contains no other critical point of any $E_\tau$,
\item each $E_\tau$ satisfies the {\em \PS{c}} condition on $U$ for all $c \in \R$.
\end{enumroman}
Then
\[
C_q(E_0,u_0) \isom C_q(E_1,u_0) \quad \forall q \ge 0.
\]
\end{proposition}

The subcritical case $p < 2^\ast$ of Theorem \ref{Theorem 6} follows from Perera and Schechter \cite[Lemma 3.2.1]{MR3012848}, so we only consider the critical case $p = 2^\ast$, whose proof requires more careful estimates. We will apply Proposition \ref{Proposition 2} to the family of functionals
\begin{multline*}
E_\tau(u) = \int_\Omega \bigg(\half\, |\nabla u|^2 - H(x,(1 - \tau)\, u + \tau\, \vartheta(u))\\[5pt]
- \frac{1}{2^\ast}\, \left[(1 - \tau)\, |u|^{2^\ast} + \tau\, |\vartheta(u)|^{2^\ast}\right]\bigg)\, dx, \quad u \in H^1_0(\Omega),\, \tau \in [0,1]
\end{multline*}
in a sufficiently small closed ball $\closure{B_\eps(0)}$, noting that $E_0 = E$ and $E_1 = \widetilde{E}$.

Let
\begin{equation} \label{19}
S = \inf_{u \in H^1_0(\Omega) \setminus \set{0}}\, \frac{\dint_\Omega |\nabla u|^2\, dx}{\left(\dint_\Omega |u|^{2^\ast}\, dx\right)^{2/2^\ast}}
\end{equation}
be the best constant in the Sobolev inequality. First we prove the following lemma on the \PS{c} sequences of $E_\tau$.

\begin{lemma} \label{Lemma 1}
Let $\tau \in [0,1]$, let $c \in \R$, and let $\seq{u_j}$ be a bounded {\em \PS{c}} sequence of $E_\tau$. Then a renamed subsequence of $\seq{u_j}$ converges weakly to a critical point $u$ of $E_\tau$. Moreover, $u_j \to u$ strongly in each of the following cases:
\begin{enumroman}
\item $\tau \in [0,1)$, $c < (1/N)\, S^{N/2}/(1 - \tau)^{(N-2)/2}$, and $u = 0$,
\item $\tau = 1$.
\end{enumroman}
\end{lemma}

\begin{proof}
We have
\begin{multline} \label{15}
E_\tau(u_j) = \int_\Omega \bigg(\half\, |\nabla u_j|^2 - H(x,(1 - \tau)\, u_j + \tau\, \vartheta(u_j))\\[5pt]
- \frac{1}{2^\ast} \left[(1 - \tau)\, |u_j|^{2^\ast} + \tau\, |\vartheta(u_j)|^{2^\ast}\right]\bigg)\, dx = c + \o(1)
\end{multline}
and
\begin{multline} \label{16}
E_\tau'(u_j)\, v = \int_\Omega \bigg(\nabla u_j \cdot \nabla v - h(x,(1 - \tau)\, u_j + \tau\, \vartheta(u_j))\, (1 - \tau + \tau\, \vartheta'(u_j))\, v\\[5pt]
- \left[(1 - \tau)\, |u_j|^{2^\ast - 2}\, u_j + \tau\, |\vartheta(u_j)|^{2^\ast - 2}\, \vartheta(u_j)\, \vartheta'(u_j)\right] v\bigg)\, dx = \o(\norm{v})
\end{multline}
for all $v \in H^1_0(\Omega)$. Since $\seq{u_j}$ is bounded, a renamed subsequence of $\seq{u_j}$ converges to some function $u$ weakly in $H^1_0(\Omega)$, strongly in $L^\sigma(\Omega)$ for all $\sigma \in [1,2^\ast)$, and a.e.\! in $\Omega$. Passing to the limit in \eqref{16} then shows that $u$ is a critical point of $E_\tau$.

Suppose $\tau \in [0,1)$, $c < (1/N)\, S^{N/2}/(1 - \tau)^{(N-2)/2}$, and $u = 0$. Then \eqref{15} and \eqref{16} with $v = u_j$ imply
\begin{equation} \label{17}
\half \int_\Omega |\nabla u_j|^2\, dx - \frac{1}{2^\ast}\, (1 - \tau) \int_\Omega |u_j|^{2^\ast}\, dx = c + \o(1)
\end{equation}
and
\begin{equation} \label{18}
\int_\Omega |\nabla u_j|^2\, dx - (1 - \tau) \int_\Omega |u_j|^{2^\ast}\, dx = \o(1),
\end{equation}
while \eqref{19} entails
\begin{equation} \label{20}
\int_\Omega |\nabla u_j|^2\, dx \ge S \left(\int_\Omega |u_j|^{2^\ast}\, dx\right)^{2/2^\ast}.
\end{equation}
A straightforward calculation combining \eqref{17}--\eqref{20} gives
\[
\left[\frac{S^{N/(N-2)}}{1 - \tau} - (Nc)^{2/(N-2)}\right] \int_\Omega |\nabla u_j|^2\, dx \le \o(1),
\]
and this implies that $u_j \to 0$ since $c < (1/N)\, S^{N/2}/(1 - \tau)^{(N-2)/2}$.

If $\tau = 1$, then we are in the subcritical case and a standard argument shows that $u_j \to u$.
\end{proof}

Next we show that there is no nonzero critical point of any $E_\tau$ in a sufficiently small closed ball around the origin.

\begin{lemma} \label{Lemma 2}
If $0$ is an isolated critical point of $E$ and $\eps > 0$ is sufficiently small, then $0$ is the only critical point of $E_\tau$ in $\closure{B_\eps(0)}$ for all $\tau \in [0,1]$.
\end{lemma}

\begin{proof}
Suppose that the conclusion of the lemma is false. Then there exist sequences $\seq{\tau_j} \subset [0,1]$ and $\seq{u_j} \subset H^1_0(\Omega) \setminus \set{0}$ such that $E_{\tau_j}'(u_j) = 0$ and $u_j \to 0$ in $H^1_0(\Omega)$. We will show that $u_j \to 0$ in $C(\closure{\Omega})$ for a renamed subsequence. Then for all sufficiently large $j$, $|u_j| \le \delta/2$ and hence $E'(u_j) = E_{\tau_j}'(u_j) = 0$, contradicting the assumption that $0$ is an isolated critical point of $E$.

We have
\begin{equation} \label{24}
\left\{\begin{aligned}
- \Delta u_j & = h_j(x,u_j) && \text{in } \Omega\\[10pt]
u_j & = 0 && \text{on } \bdry{\Omega},
\end{aligned}\right.
\end{equation}
where
\begin{multline*}
h_j(x,t) = h(x,(1 - \tau_j)\, t + \tau_j\, \vartheta(t))\, (1 - \tau_j + \tau_j\, \vartheta'(t))\\[5pt]
+ (1 - \tau_j)\, |t|^{2^\ast - 2}\, t + \tau_j\, |\vartheta(t)|^{2^\ast - 2}\, \vartheta(t)\, \vartheta'(t)
\end{multline*}
satisfies a growth condition
\begin{equation} \label{23}
|h_j(x,t)| \le C(|t|^{2^\ast - 1} + 1) \quad \forall (x,t) \in \Omega \times \R
\end{equation}
for some $C > 0$ independent of $j$. First we show that $\seq{u_j}$ is bounded in $C(\closure{\Omega})$. By de Figueiredo et al.\! \cite[Proposition 3.7]{MR2530603}, it suffices to show that
\begin{equation} \label{22}
\int_A |u_j|^{2^\ast}\, dx \to 0 \quad \text{as } |A| \to 0
\end{equation}
uniformly in $j$. Suppose not. Then there exist $\eps_0 > 0$ and a sequence $\seq{A_k}$ of subsets of $\Omega$ with $|A_k| \to 0$ such that
\begin{equation} \label{21}
\int_{A_k} |u_{j_k}|^{2^\ast}\, dx \ge \eps_0
\end{equation}
for some $j_k$. If the sequence $\seq{j_k}$ is bounded, then there exists $j_0$ such that, for a renamed subsequence, $j_k \equiv j_0$ and hence
\[
\int_{A_k} |u_{j_k}|^{2^\ast}\, dx = \int_{A_k} |u_{j_0}|^{2^\ast}\, dx \to 0 \quad \text{as } k \to \infty
\]
since $|A_k| \to 0$, contradicting \eqref{21}. On the other hand, if $\seq{j_k}$ is unbounded then, for a renamed subsequence, $j_k \to \infty$ as $k \to \infty$ and hence
\[
\int_{A_k} |u_{j_k}|^{2^\ast}\, dx \le \int_\Omega |u_{j_k}|^{2^\ast}\, dx \to 0 \quad \text{as } k \to \infty
\]
since $u_j \to 0$ in $H^1_0(\Omega)$, again contradicting \eqref{21}. So \eqref{22} holds and hence $\seq{u_j}$ is bounded in $C(\closure{\Omega})$.

Next we note that from $u_j \to 0$ in $H^1_0(\Omega)$ it follows, for a renamed subsequence, $u_j \to 0$ a.e., whence $h_j(x,u_j) \to h(x,0) = 0$ a.e. Moreover, $h_j(x,u_j)$ is bounded, because$\seq{u_j}$ is bounded in $C(\closure{\Omega})$ and $h_j$ satisfies \eqref{23}. Thus, $h_j(x,u_j) \to 0$ in $L^q(\Omega)$ for any $q\in [1,\infty)$. Since $u_j$ solves \eqref{24}, then $u_j \to 0$ in $W^{2,q}(\Omega)$ by the Calder\'{o}n-Zygmund inequality. The continuous embedding $W^{2,q}(\Omega) \hookrightarrow C(\closure{\Omega})$, $q> N/2$, entails $u_j \to 0$ in $C(\closure{\Omega})$, as desired.
\end{proof}

We are now ready to prove Theorem \ref{Theorem 6}.

\begin{proof}[Proof of Theorem \ref{Theorem 6}]
We apply Proposition \ref{Proposition 2} to the family of functionals $E_\tau,\, \tau \in [0,1]$ in $\closure{B_\eps(0)}$, where $\eps > 0$ is as in Lemma \ref{Lemma 2}, so that $\closure{B_\eps(0)}$ contains no nonzero critical point of any $E_\tau$. Clearly, the map $[0,1] \to C^1(\closure{B_\eps(0)},\R),\, \tau \mapsto E_\tau$ is continuous. We will show that each $E_\tau$ satisfies the \PS{c} condition on $\closure{B_\eps(0)}$ for all $c \in \R$ if $\eps$ is further restricted.

First we note that if $\eps > 0$ is sufficiently small, then
\begin{equation} \label{26}
\sup_{u \in \closure{B_\eps(0)}}\, E_\tau(u) < \frac{1}{N}\, S^{N/2} \quad \forall \tau \in [0,1].
\end{equation}
To see this, suppose that there exist sequences $\seq{\tau_j} \subset [0,1]$ and $u_j \to 0$ in $H^1_0(\Omega)$ such that
\begin{equation} \label{25}
E_{\tau_j}(u_j) \ge \frac{1}{N}\, S^{N/2}.
\end{equation}
For a renamed subsequence, $u_j \to 0$ a.e.\! in $\Omega$ and strongly in $L^q(\Omega)$ for all $q\in [1,2^\ast]$. Since $h$ satisfies \eqref{11}, then $E_{\tau_j}(u_j) \to 0$, contradicting \eqref{25}. So \eqref{26} holds.

Now let $\eps > 0$ be as above, let $\tau \in [0,1]$, let $c \in \R$, and let $\seq{u_j} \subset \closure{B_\eps(0)}$ be a \PS{c} sequence of $E_\tau$. By Lemma \ref{Lemma 1}, a renamed subsequence of $\seq{u_j}$ converges weakly to a critical point $u \in \closure{B_\eps(0)}$ of $E_\tau$. Since $\closure{B_\eps(0)}$ contains no nonzero critical point of $E_\tau$, $u = 0$. Since $u_j \in \closure{B_\eps(0)}$ and $E_\tau(u_j) \to c$,
\[
c \le \sup_{u \in \closure{B_\eps(0)}}\, E_\tau(u),
\]
so \eqref{26} implies
\[
c < \frac{1}{N}\, S^{N/2} \le \frac{1}{N}\, \frac{S^{N/2}}{(1 - \tau)^{(N-2)/2}}
\]
if $\tau \in [0,1)$. So $u_j \to 0$ strongly by Lemma \ref{Lemma 1} and hence $E_\tau$ satisfies the \PS{c} condition on $\closure{B_\eps(0)}$.
\end{proof}

Finally we prove the proposition below, from which Theorem \ref{Theorem 4} is immediate. Indeed, if \eqref{13} holds, then
\[
H(x,t) = \frac{\mu}{\sigma}\, |t|^\sigma + \o(|t|^\sigma), \hquad 2H(x,t) - th(x,t) = \left(\frac{2}{\sigma} - 1\right) \mu\, |t|^\sigma + \o(|t|^\sigma)\text{ as } t \to 0
\]
uniformly a.e. in $\Omega$, and \eqref{30} and \eqref{29} follow from this since $\mu > 0$ and $\sigma \in (1,2)$.

\begin{proposition} \label{Proposition 3}
Assume that \eqref{11} holds,
\begin{equation} \label{30}
\lim_{t \to 0}\, \frac{H(x,t)}{t^2} = + \infty \quad \text{uniformly a.e. in $\Omega$},
\end{equation}
and
\begin{equation} \label{29}
2H(x,t) - th(x,t) > \left(1 - \frac{2}{p}\right) |t|^p \quad \forall x \in \Omega,\, 0 < |t| \le \delta
\end{equation}
for some $\delta > 0$. If $0$ is an isolated critical point of $E$, then
\[
C_q(E,0) = 0 \quad \forall q \ge 0.
\]
\end{proposition}

First we prove a localization result for subcritical problems.

\begin{lemma} \label{Lemma 3}
For $i = 0, 1$, let $h_i$ be Carath\'{e}odory functions on $\Omega \times \R$ satisfying
\[
|h_i(x,t)| \le C(|t|^{r-1} + 1) \quad \forall (x,t) \in \Omega \times \R
\]
for some $C > 0$ and $r > 2$ if $N = 2$ and $2 < r < 2^\ast$ if $N \ge 3$, and set
\[
E_i(u) = \int_\Omega \left(\half\, |\nabla u|^2 - H_i(x,u)\right) dx, \quad u \in H^1_0(\Omega),
\]
where $H_i(x,t) = \int_0^t h_i(x,s)\, ds$. If $0$ is an isolated critical point of $E_0$ and
\begin{equation} \label{27}
h_0(x,t) = h_1(x,t) \quad \forall x \in \Omega,\, |t| \le \delta
\end{equation}
for some $\delta > 0$, then $0$ is also an isolated critical point of $E_1$ and
\[
C_q(E_0,0) \isom C_q(E_1,0) \quad \forall q \ge 0.
\]
\end{lemma}

\begin{proof}
We apply Proposition \ref{Proposition 2} to the family of functionals
\[
E_\tau(u) = \int_\Omega \left(\half\, |\nabla u|^2 - (1 - \tau)\, H_0(x,u) - \tau\, H_1(x,u)\right) dx, \quad u \in H^1_0(\Omega),\, \tau \in [0,1]
\]
in a small closed ball $\closure{B_\eps(0)}$. Clearly, the map $[0,1] \to C^1(\closure{B_\eps(0)},\R),\, \tau \mapsto E_\tau$ is continuous. Since $r < 2^\ast$ if $N \ge 3$, each $E_\tau$ satisfies the \PS{c} condition on $\closure{B_\eps(0)}$ for all $c \in \R$. We will show that $\closure{B_\eps(0)}$ contains no nonzero critical point of any $E_\tau$ if $\eps > 0$ is sufficiently small.

Suppose that there exist sequences $\seq{\tau_j} \subset [0,1]$ and $\seq{u_j} \subset H^1_0(\Omega) \setminus \set{0}$ such that $E_{\tau_j}'(u_j) = 0$ and $u_j \to 0$ in $H^1_0(\Omega)$. We will show that $u_j \to 0$ in $C(\closure{\Omega})$ for a renamed subsequence. Then for all sufficiently large $j$, $|u_j| \le \delta$ and hence $E_0'(u_j) = E_{\tau_j}'(u_j) = 0$ by \eqref{27}, contradicting the assumption that $0$ is an isolated critical point of $E_0$.

We have
\begin{equation} \label{35}
\left\{\begin{aligned}
- \Delta u_j & = h_j(x,u_j) && \text{in } \Omega\\[10pt]
u_j & = 0 && \text{on } \bdry{\Omega},
\end{aligned}\right.
\end{equation}
where $h_j(x,t) = (1 - \tau_j)\, h_0(x,t) + \tau_j\, h_1(x,t)$ also satisfies the growth condition
\begin{equation} \label{36}
|h_j(x,t)| \le C(|t|^{r-1} + 1) \quad \forall (x,t) \in \Omega \times \R.
\end{equation}
For a renamed subsequence, $u_j \to 0$ a.e.\! and $\tau_j$ converges to some $\tau \in [0,1]$, and hence
\[
h_j(x,u_j) \to (1 - \tau)\, h_0(x,0) + \tau\, h_1(x,0) = 0 \quad \text{a.e.}
\]
So, if $u_j \to 0$ in $L^q(\Omega)$, then $h_j(x,u_j) \to 0$ in $L^{q/(r-1)}(\Omega)$ by \eqref{36}. Since $u_j$ solves \eqref{35}, then $u_j \to 0$ in $W^{2,q/(r-1)}(\Omega)$ by the Calder\'{o}n-Zygmund inequality. Then $u_j \to 0$ in $L^{Nq/(N(r-1)-2q)}(\Omega)$ also by the Sobolev embedding theorem. Starting with $q = 2^\ast$, iterating until $q > N(r - 1)/2$, and using the continuous embedding $W^{2,q/(r-1)}(\Omega) \hookrightarrow C(\closure{\Omega})$ now entails $u_j \to 0$ in $C(\closure{\Omega})$, as desired.
\end{proof}

We are now ready to prove Proposition \ref{Proposition 3}.

\begin{proof}[Proof of Proposition \ref{Proposition 3}]
By Theorem \ref{Theorem 6}, $0$ is also an isolated critical point of $\widetilde{E}$ and
\begin{equation} \label{32}
C_q(E,0) \isom C_q(\widetilde{E},0) \quad \forall q \ge 0,
\end{equation}
where $\widetilde{E}$ is as in that theorem. Let
\[
\bar{h}(x,t) = \begin{cases}
-h(x,- \delta)\, t/\delta, & t < - \delta\\[5pt]
h(x,t), & |t| \le \delta\\[5pt]
h(x,\delta)\, t/\delta, & t > \delta,
\end{cases} \qquad k(t) = \begin{cases}
|t|^{p-2}\, t, & |t| \le \delta\\[5pt]
\delta^{p-2}\, t, & |t| > \delta,
\end{cases}
\]
and set
\[
\bar{E}(u) = \int_\Omega \left(\half\, |\nabla u|^2 - \bar{H}(x,u) - K(u)\right) dx, \quad u \in H^1_0(\Omega),
\]
where $\bar{H}(x,t) = \int_0^t \bar{h}(x,s)\, ds$ and $K(t) = \int_0^t k(s)\, ds$. Then $0$ is also an isolated critical point of $\bar{E}$ and
\begin{equation}
C_q(\widetilde{E},0) \isom C_q(\bar{E},0) \quad \forall q \ge 0
\end{equation}
by Lemma \ref{Lemma 3}. Since $\bar{E}(0) = 0$,
\begin{equation} \label{34}
C_q(\bar{E},0) = H_q(\bar{E}^0 \cap B,\bar{E}^0 \cap B \setminus \set{0}),
\end{equation}
where $\bar{E}^0 = \{u \in H^1_0(\Omega) : \bar{E}(u) \le 0\}$ and $B = \set{u \in H^1_0(\Omega) : \norm{u} \le 1}$. We will show that $\bar{E}^0 \cap B$ is contractible to $0$ and $\bar{E}^0 \cap B \setminus \set{0}$ is a strong deformation retract of $B \setminus \set{0} \homo \bdry{B} =: S$. Since $H^1_0(\Omega)$ is infinite dimensional and hence $S$ is contractible, the conclusion will then follow from \eqref{32}--\eqref{34}.

For $u \in S$ and $0 < t \le 1$,
\[
\bar{E}(tu) = \frac{t^2}{2} - \int_\Omega \left(\bar{H}(x,tu) + K(tu)\right) dx.
\]
Since $K \ge 0$,
\[
\bar{E}(tu) \le t^2 \left(\half - \int_\Omega \frac{\bar{H}(x,tu)}{t^2}\, dx\right).
\]
Since $\bar{H} = H$ on $\Omega \times [- \delta,\delta]$, \eqref{30} implies
\[
\lim_{t \to 0}\, \frac{\bar{H}(x,t)}{t^2} = + \infty \quad \text{uniformly a.e. in $\Omega$},
\]
so
\[
\lim_{t \to 0}\, \int_\Omega \frac{\bar{H}(x,tu)}{t^2}\, dx = + \infty
\]
by Fatou's lemma. So $\bar{E}(tu) < 0$ for all sufficiently small $t > 0$ (depending on $u$). Moreover,
\[
\frac{d}{dt}\, \bar{E}(tu) = t - \int_\Omega \bar{h}(x,tu)\, u\, dx - \int_\Omega k(tu)\, u\, dx = \frac{2}{t}\, \bar{E}(tu) + \frac{1}{t} \int_\Omega P(x,tu)\, dx,
\]
where $P(x,t) = 2 \bar{H}(x,t) - t \bar{h}(x,t) + 2K(t) - tk(t)$. A straightforward calculation shows that
\[
P(x,t) = \begin{cases}
Q(x,- \delta), & t < - \delta\\[5pt]
Q(x,t), & |t| \le \delta\\[5pt]
Q(x,\delta), & t > \delta,
\end{cases}
\]
where
\[
Q(x,t) = 2H(x,t) - th(x,t) - \left(1 - \frac{2}{p}\right) |t|^p > 0 \quad \forall x \in \Omega,\, 0 < |t| \le \delta
\]
by \eqref{29}, so $P > 0$ on $\Omega \times (\R \setminus \set{0})$. So $\dfrac{d}{dt}\, \bar{E}(tu) > \dfrac{2}{t}\, \bar{E}(tu)$ and hence
\begin{equation} \label{33}
\bar{E}(tu) \ge 0 \implies \dfrac{d}{dt}\, \bar{E}(tu) > 0.
\end{equation}
Thus, there is a unique $0 < T(u) \le 1$ such that $E(tu) < 0$ for $0 < t < T(u)$, $E(T(u)\, u) \le 0$, and $E(tu) > 0$ for $T(u) < t \le 1$. We claim that the map $T : S \to (0,1]$ is continuous. By \eqref{33} and the implicit function theorem, $T$ is $C^1$ on $\bgset{u \in S : T(u) < 1}$, so it suffices to show that if $u_j \to u$ and $T(u) = 1$, then $T(u_j) \to 1$. But for any $t < 1$, $\bar{E}(tu_j) \to \bar{E}(tu)$ and $\bar{E}(tu) < 0$ since $t < T(u)$, so for all sufficiently large $j$, $\bar{E}(tu_j) < 0$ and hence $T(u_j) \ge t$. So $T(u_j) \to 1$.

Thus,
\[
\bar{E}^0 \cap B = \bgset{tu : u \in S,\, 0 \le t \le T(u)}
\]
and is radially contractible to $0$, and
\begin{align*}
(B \setminus \set{0}) \times [0,1] & \to B \setminus \set{0},\\[5pt]
(u,t) & \mapsto \begin{cases}
(1 - t)\, u + t\, T(\pi(u))\, \pi(u), & u \in B \setminus \bar{E}^0\\[5pt]
u, & u \in \bar{E}^0 \cap B \setminus \set{0},
\end{cases}
\end{align*}
where $\pi$ is the radial projection onto $S$, is a strong deformation retraction of $B \setminus \set{0}$ onto $\bar{E}^0 \cap B \setminus \set{0}$.
\end{proof}

\section{Proofs of Theorems \ref{Theorem 1} and \ref{Theorem 2}} \label{Proofs}

First we prove Theorem \ref{Theorem 2}.

\begin{proof}[Proof of Theorem \ref{Theorem 2}]
We apply Theorem \ref{Theorem 3} to the functional
\[
E(u) = \int_\Omega \left(\half\, |\nabla u|^2 - \mu F(x,u) - G(x,u) - \frac{1}{p}\, |u|^p\right) dx, \quad u \in X = H^1_0(\Omega).
\]
We may assume that $E$ has only a finite number of critical points, and hence they are all isolated, since otherwise there is nothing to prove. Since $h(x,t) = \mu f(x,t) + g(x,t)$ satisfies \eqref{11} and \eqref{13} by \eqref{3}, \eqref{2}, \eqref{40}, and \eqref{5}, then
\begin{equation} \label{41}
C_q(E,0) = 0 \quad \forall q \ge 0
\end{equation}
by Theorem \ref{Theorem 4}.

We have the orthogonal direct sum decomposition $X = Y \oplus Z,\, u = v + w$, where $Y$ is spanned by the eigenfunctions associated with $\lambda_1, \dots, \lambda_l$ and $Z$ by those associated with $\lambda_{l+1}, \lambda_{l+2}, \dots$. Let $z_0 \in X \setminus Y$ and $0 < \rho < R$, and set
\begin{equation} \label{55}
A = \set{u = v + t z_0 : v \in Y,\, t \ge 0,\, \norm{u} \le R}, \qquad B = \set{w \in Z : \norm{w} \le \rho}.
\end{equation}
First we show that
\begin{equation} \label{39}
\sup_{\bdry{A}}\, E = 0 < \inf_{\bdry{B}}\, E
\end{equation}
if $R$ is sufficiently large and $\rho$ and $\mu$ are sufficiently small. By \eqref{37} and \eqref{10},
\[
E(v) \le \half \int_\Omega \left(|\nabla v|^2 - \lambda_l\, v^2\right) dx \le 0 \quad \forall v \in Y.
\]
Moreover, $E$ is anticoercive on the finite dimensional subspace $Y \oplus \R z_0$ of $X$ since $p > r > 2$. Since $E(0) = 0$, it follows that the equality in \eqref{39} holds for sufficiently large $R$. Denoting by $C$ a generic positive constant,
\[
F(x,t) \le C(|t|^r + 1), \quad G(x,t) \le \half\, \lambda\, t^2 + C |t|^r \quad \forall (x,t) \in \Omega \times \R
\]
by \eqref{2}, \eqref{5}, and \eqref{9}, and hence
\begin{multline*}
E(w) \ge \half \int_\Omega \left(|\nabla w|^2 - \lambda\, w^2\right) dx - C \int_\Omega \left(\mu + |w|^r + |w|^p\right) dx\\[5pt]
\ge \half \left(1 - \frac{\lambda}{\lambda_{l+1}}\right) \norm{w}^2 - C \left(\mu + \norm{w}^r + \norm{w}^p\right) \quad \forall w \in Z
\end{multline*}
since $r < p < 2^\ast$ when $N \ge 3$. Since $\lambda < \lambda_{l+1}$ and $2 < r < p$, it follows that there exist $\rho, \mu^\ast > 0$ such that the inequality in \eqref{39} holds for all $\mu \in (0,\mu^\ast)$.

Since $N = 2$, or $N \ge 3$ and $p < 2^\ast$, $E$ satisfies the \PS{c} condition for all $c \in \R$, so Theorem \ref{Theorem 3} now gives a pair of critical points $u_1, u_2$ of $E$ with
\[
C_l(E,u_1) \ne 0, \qquad C_{l+1}(E,u_2) \ne 0.
\]
They are nontrivial in view of \eqref{41}.
\end{proof}

In preparation for the proof of Theorem \ref{Theorem 1}, next we determine an energy range where the associated functional
\[
E_\lambda(u) = \int_\Omega \left(\half\, |\nabla u|^2 - \mu F(x,u) - \frac{\lambda}{2}\, u^2 - \frac{1}{2^\ast}\, |u|^{2^\ast}\right) dx, \quad u \in X = H^1_0(\Omega)
\]
satisfies the \PS{c} condition.

\begin{lemma} \label{Lemma 4}
Let $\mu \in (0,1)$ and let $S$ be as in \eqref{19}. Then there exists a constant $\kappa > 0$ such that $E_\lambda$ satisfies the {\em \PS{c}} condition for all
\begin{equation} \label{52}
c < \frac{1}{N}\, S^{N/2} - \kappa \mu.
\end{equation}
\end{lemma}

\begin{proof}
Let $c \in \R$ and let $\seq{u_j}$ be a \PS{c} sequence of $E_\lambda$, so that
\begin{equation} \label{42}
E_\lambda(u_j) = \int_\Omega \left(\half\, |\nabla u_j|^2 - \mu F(x,u_j) - \frac{\lambda}{2}\, u_j^2 - \frac{1}{2^\ast}\, |u_j|^{2^\ast}\right) dx = c + \o(1)
\end{equation}
and
\begin{equation} \label{43}
E_\lambda'(u_j)\, v = \int_\Omega \left(\nabla u_j \cdot \nabla v - \mu f(x,u_j)\, v - \lambda u_j\, v - |u_j|^{2^\ast - 2}\, u_j\, v\right) dx = \o(\norm{v})
\end{equation}
for all $v \in H^1_0(\Omega)$. Taking $v = u_j$ in \eqref{43} gives
\begin{equation} \label{44}
\int_\Omega \left(|\nabla u_j|^2 - \mu f(x,u_j)\, u_j - \lambda u_j^2 - |u_j|^{2^\ast}\right) dx = \o(\norm{u_j}).
\end{equation}
Since $r < 2^\ast$, \eqref{42} and \eqref{44} imply that $\seq{u_j}$ is bounded, so a renamed subsequence of $\seq{u_j}$ converges to some function $u$ weakly in $H^1_0(\Omega)$, strongly in $L^s(\Omega)$ for all $s \in [1,2^\ast)$, and a.e.\! in $\Omega$. Setting $\widetilde{u}_j = u_j - u$, we will show that $\widetilde{u}_j \to 0$ in $H^1_0(\Omega)$.

Equation \eqref{44} implies
\begin{equation} \label{48}
\norm{u_j}^2 = \pnorm[2^\ast]{u_j}^{2^\ast} + \int_\Omega \left(\mu f(x,u)\, u + \lambda u^2\right) dx + \o(1),
\end{equation}
where $\pnorm[2^\ast]{\cdot}$ denotes the $L^{2^\ast}(\Omega)$-norm. Taking $v = u$ in \eqref{43} and passing to the limit gives
\begin{equation} \label{46}
\norm{u}^2 = \pnorm[2^\ast]{u}^{2^\ast} + \int_\Omega \left(\mu f(x,u)\, u + \lambda u^2\right) dx.
\end{equation}
Since
\begin{equation} \label{54}
\norm{\widetilde{u}_j}^2 = \norm{u_j}^2 - \norm{u}^2 + \o(1)
\end{equation}
and
\[
\pnorm[2^\ast]{\widetilde{u}_j}^{2^\ast} = \pnorm[2^\ast]{u_j}^{2^\ast} - \pnorm[2^\ast]{u}^{2^\ast} + \o(1)
\]
by the Br{\'e}zis-Lieb lemma \cite[Theorem 1]{MR699419}, \eqref{48}, \eqref{46}, and \eqref{19} imply
\[
\norm{\widetilde{u}_j}^2 = \pnorm[2^\ast]{\widetilde{u}_j}^{2^\ast} + \o(1) \le \frac{\norm{\widetilde{u}_j}^{2^\ast}}{S^{2^\ast/2}} + \o(1),
\]
so
\begin{equation} \label{50}
\norm{\widetilde{u}_j}^2 \left(S^{N/(N-2)} - \norm{\widetilde{u}_j}^{4/(N-2)}\right) \le \o(1).
\end{equation}
On the other hand, \eqref{42} implies
\[
c = \half \norm{u_j}^2 - \frac{1}{2^\ast} \pnorm[2^\ast]{u_j}^{2^\ast} - \int_\Omega \left(\mu F(x,u) + \frac{\lambda}{2}\, u^2\right) dx + \o(1),
\]
and a straightforward calculation combining this with \eqref{48}--\eqref{54} gives
\begin{equation} \label{53}
c = \frac{1}{N} \norm{\widetilde{u}_j}^2 + \int_\Omega K(x,u)\, dx + \o(1),
\end{equation}
where
\[
K(x,t) = \frac{1}{N}\, |t|^{2^\ast} + \mu \left(\half\, f(x,t)\, t - F(x,t)\right).
\]
Denoting by $C$ a generic positive constant,
\[
K(x,t) \ge \frac{1}{N}\, |t|^{2^\ast} - \mu\, C(|t|^r + 1) \quad \forall (x,t) \in \Omega \times \R
\]
by \eqref{2}, and since $r < 2^\ast$, minimizing the right-hand side over $t$ and using $\mu \in (0,1)$ gives $K(x,t) \ge - \mu\, C$. So \eqref{53} implies
\[
\norm{\widetilde{u}_j}^2 \le N(c + \kappa \mu) + \o(1)
\]
for some constant $\kappa > 0$, and \eqref{50} together with this implies that $\widetilde{u}_j \to 0$ when \eqref{52} holds.
\end{proof}

We are now ready to prove Theorem \ref{Theorem 1}.

\begin{proof}[Proof of Theorem \ref{Theorem 1}]
Since $\lambda \ge \lambda_1$, we have $\lambda_l \le \lambda < \lambda_{l+1}$ for some $l \ge 1$. We proceed as in the proof of Theorem \ref{Theorem 2}, but will choose $z_0 \in X \setminus Y$ so that
\begin{equation} \label{56}
\sup_A\, E_\lambda < \frac{1}{N}\, S^{N/2} - \kappa \mu,
\end{equation}
where $A$ is as in \eqref{55} and $\kappa$ is as in Lemma \ref{Lemma 4}. The desired conclusion will then follow from Theorems \ref{Theorem 3} and \ref{Theorem 4} as before.

When $\Omega = \R^N$, the infimum in \eqref{19} is attained by the family of functions
\[
U_\eps(x) = \frac{[N(N - 2)\, \eps]^{(N-2)/4}}{\left(\eps + |x|^2\right)^{(N-2)/2}}, \quad \eps > 0
\]
(see Br{\'e}zis and Nirenberg \cite{MR709644}). We may assume without loss of generality that the ball $B_1(0) \subset \Omega$. Let $\eta \in C^\infty_0(B_1(0))$ be a cutoff function such that $\eta = 1$ on $B_{1/2}(0)$, set
\[
u_\eps(x) = \eta(x)\, U_\eps(x), \quad x \in \Omega,
\]
and let
\[
z_0 = \begin{cases}
u_\eps, & \lambda > \lambda_l\\[5pt]
u_\eps - P_l\, u_\eps, & \lambda = \lambda_l,
\end{cases}
\]
where $P_l$ denotes the projection onto the eigenspace of $\lambda_l$. By \eqref{37},
\[
\sup_{u \in A}\, E_\lambda(u) \le \sup_{u \in A}\, \int_\Omega \left(\half\, |\nabla u|^2 - \frac{\lambda}{2}\, u^2 - \frac{1}{2^\ast}\, |u|^{2^\ast}\right) dx,
\]
and by Capozzi et al.\! \cite[Lemma 2.5]{MR831041}, the supremum on the right-hand side is less than $(1/N)\, S^{N/2}$ if $\eps > 0$ is sufficiently small. So \eqref{56} holds for sufficiently small $\mu$.
\end{proof}

\section*{Acknowledgement}
Work performed within the research project of MIUR Prin 2017 `Nonlinear Differential Problems via Variational, Topological and Set-valued Methods' (Grant No. 2017AYM8XW) and partially supported by INDAM/GNAMPA.

\def\cdprime{$''$}

\end{document}